\documentclass[14pt, reqno]{amsart}
\usepackage{amsmath,amssymb, amsthm}
\usepackage[colorlinks,urlcolor=red]{hyperref}
\usepackage{graphicx}
 \theoremstyle{plain}
\theoremstyle{plain}
\newtheorem{theorem}{Theorem}[section]
\newtheorem{lemma}[theorem]{Lemma}

\newtheorem{definition}[theorem]{Definition}

\allowdisplaybreaks

\numberwithin{equation}{section}

\renewcommand{\to}{\longrightarrow}
\begin{document}

\title[convergence of three-step iterative process   ]
{ strong convergence of three-step iterative process with errors
for three multivalued mappings}
\author[Mohammad Eslamian]{ M. Eslamian,  S. Homaeipour }
\date{}
\thanks{  }
\thanks{\it Email : {\rm   mhmdeslamian@gmail.com    }}
\maketitle

\vspace*{-0.5cm}



\begin{center}
{\footnotesize  Department of Mathematics, IKIU, Qazvin 34149,
Iran. }
\end{center}

\hrulefill

{\footnotesize \noindent {\bf Abstract.} In this paper, we
introduced a three-step iterative process with errors for three
multivalued mappings satisfying the condition (C) in uniformly
convex Banach spaces and establish  strong convergence theorems
for the proposed process under some basic boundary conditions.
Our results generalized recent known results in the literature.

\vskip0.5cm

\noindent {\bf Keywords}:Fixed point, Condition (C) , Three-step
iteration process, strong convergence .\\ \noindent {\bf AMS
Subject Classification}: 47H10, 47H09. }

\hrulefill

\section{ Introduction }\label{S:1}
The study of fixed points for multivalued contractions and
nonexpansive mappings using the Hausdorff metric was initiated by
Markin \cite{ma} and Nadler \cite{na}. Since then the metric
fixed point theory of multivalued mappings has been rapidly
developed. The theory of multivalued mappings has applications in
control theory, convex optimization, differential equations and
economics. Theory of multivalued nonexpansive mappings is harder
than the corresponding theory of singlevalued nonexpansive
mappings. Different iterative processes have been used to
approximate fixed points of multivalued nonexpansive mappings. In
particular in 2005, Sastry and Babu \cite{sa} proved that the Mann
and Ishikawa iteration process for multivalued maping $T$ with a
fixed point $p$ converge to a fixed point $q$ of $T$ under certain
conditions. They also claimed that the fixed point $ q$ may be
different from $p$. Panyanak \cite{pa} extended result of Sastry
and Babu \cite{sa} to uniformly convex Banach spaces. Recently,
Song and Wang  \cite {so1} noted that there was a gap in the
proof of the main result in \cite{pa}. They further revised the
gap and also gave the affirmative answer to Panyanak's open
question.  Shahzad and Zegeye \cite{sh} extended and improved
results already appeared in the papers \cite{pa,sa,so1}. Very
recently, motivated by \cite{sh}, Cholamjiak and Suantai
\cite{c1,c2} introduced some new two-step iterative process for
two multivalued mappings in Banach spaces and prove strong
convergence of the proposed iterations.\par Glowinski and Le
Tallec \cite{gl} used three-step iterative process to find the
approximate solutions of the elastoviscoplasticity problem,
liquid crystal theory, and eigenvalue computation. It has been
shown in \cite{gl} that the three-step iterative process gives
better numerical results than the two-step and one-step
approximate iterations. In 1998, Haubruge et al. \cite{ha}
studied the convergence analysis of three-step process of
Glowinski and Le Tallec \cite{gl} and applied these process to
obtain new splitting-type algorithms for solving variation
inequalities, separable convex programming and minimization of a
sum of convex functions. They also proved that three-step
iterations lead to highly parallelized algorithms under certain
conditions. Thus we conclude that three-step process plays an
important and significant part in solving various problems, which
arise in pure and applied sciences.\par
  Now the aim of this paper is to introduce a three-step
iterative process  with errors for multivalued mappings
satisfying condition (C) and then prove some strong convergence
theorems for such process in uniformly convex Banach space. Both
Mann and Ishikawa iterative processes for multivalued mappings
can be obtained from this process as special cases by suitably
choosing the parameters. Our results generalized recent known
result in literature.
\section{Preliminaries}
Recall that a Banach space $X$ is said to be uniformly convex if
for each $t\in [0,2]$, the modulus of convexity of $X$ given
by:$$\delta (t)=inf\{1-\frac{1}{2}\|x+y\|: \|x\|\leq1, \|y\|\leq1
,\|x-y\|\geq t\}$$ satisfies the inequality $\delta(t)>0$ for all
$t>0.$ A Banach space $X$ is said to satisfy Opial's condition if
$x_{n}\to z$  weakly as $n\to\infty$  and $z\neq y$ imply that
$$\lim sup_{n\to\infty}\|x_{n}-z\|< \lim
sup_{n\to\infty}\|x_{n}-y\|.$$ All Hilbert spaces, all finite
dimensional Banach spaces and $\ell^{p} (1\leq p<\infty)$ have the
Opial property.\par
 A subset $E\subset X$ is called proximal if for each $x\in X$,
there exists an element $y\in E$ such that
$$\parallel x-y\parallel=dist(x,E)=inf\{\parallel x-z\parallel: z\in E\}.$$
It is known that every closed convex subset of a uniformly convex
Banach space is proximal.\\ We denote by $CB(E),K(E)$ and
 $P(E)$ the collection of all nonempty closed bounded subsets, nonempty compact
subsets, and nonempty proximal bounded subsets of $E$
respectively. The Hausdorff metric $H$ on $CB(X)$ is defined by
$$H(A,B):=\max \{\sup_{x\in A}dist(x,B),\sup_{y\in B}dist(y,A)\},$$for
all $A,B\in CB(X).$\\ Let $T:X\to 2^{X}$ be a multivalued mapping.
An element $x\in X$ is said to be a fixed point of $T$, if  $x\in
Tx$. The set of fixed points of $T$ will be denote by $F(T)$.

\begin{definition}
A multivalued mapping $T:X\to CB(X)$ is  called \\
(i) nonexpansive if $$H(Tx,Ty)\le \|x-y\|,\quad x,y\in X.$$ (ii)
quasi nonexpansive if $F(T)\neq \emptyset$ and $H(Tx,Tp)\le
\parallel x-p\parallel$ for all $x\in X$ and all $p\in
F(T).$\end{definition}
 In 2008, Suzuki
\cite{suz} introduced a condition  on mappings, called (C) which
is weaker than nonexpansiveness and stronger than quasi
nonexpansiveness. Very recently, Abkar and Eslamian \cite{ab} used
a modified Suzuki condition for multivalued mappings as follows:

\begin{definition}
A multivalued mapping $T:X\to CB(X)$ is said to satisfy condition
(C) provided that
$$ \frac{1}{2}dist(x,Tx)\le \|x-y\|\implies H(Tx,Ty)\le \|x-y\|,\quad
x,y\in X.$$ \end{definition}
\begin{lemma}(\cite{ab}) Let $T:X\to CB(X)$ be a multivalued  nonexpansive
mapping, then $T$ satisfies the condition (C).\end{lemma}
 \begin{lemma}(\cite{mh}) Let $T:X\to CB(X)$ be a multivalued mapping
 which satisfies the condition (C) and has a fixed point. Then T is a
quasi nonexpansive mapping. \end{lemma}

\begin{lemma} (\cite{mh}) Let $E$ be a nonempty subset of a Banach space $X.$
Suppose $T:E\to P(E)$ satisfies condition (C) then \\
$$ H(Tx,Ty)\leq 2 dist(x,Tx)+\|x-y\|,$$ holds for all $x,y\in
E.$
\end{lemma}

\begin{lemma}(\cite{tan}, Lemma1) Let $\{a_{n}\}, \{b_{n}\}$ and $\{\delta_{n}\}$ be
sequence of nonnegative real numbers satisfying the inequality
$$a_{n+1}\leq(1+\delta_{n})a_{n}+b_{n}.$$ If
$ \sum_{n=1}^{\infty}\delta_{n}<\infty$ and $
\sum_{n=1}^{\infty}b_{n}<\infty$, then $\lim_{n\to\infty}a_{n}$
exists. In particular , if $\{a_{n}\}$ has a subsequence
converging to 0, then $\lim_{n\to\infty}a_{n}=0.$
\end{lemma}
The following Lemma can be found in (\cite{noor}, Lemma 1.4)
\begin{lemma}Let $X$ ba a uniformly convex Banach space
and let $B_{r}(0)=\{x\in X: \parallel x\parallel\leq r\}$, $r>0$.
Then there exist a continuous, strictly increasing, and convex
function $\varphi:[0,\infty)\to [0,\infty)$ with $\varphi(0)=0$
such that
$$\parallel\alpha x+\beta y+\gamma z+\eta w
\parallel^{2}\leq\alpha\parallel x
\parallel^{2} +\beta\parallel y \parallel^{2}+\gamma\parallel z
\parallel^{2}+\eta\| w\|^{2}
-\alpha\beta \varphi(\parallel x-y
\parallel),$$ for all $x,y,z,w \in B_{r}(0)$, and $\alpha, \beta, \gamma, \eta\in[0,1]$ with
$\alpha+\beta+\gamma+\eta=1.$
\end{lemma}
\section{Main Results}
In this section we use the following iteration process.\\
(A) Let $X$ be a Banach space, $E$ be a nonempty convex subset of
$X$ and \\ $T_{1},T_{2},T_{3}:E\to CB(E)$ be three given mappings.
Then, for $x_{1}\in E$, we consider the following iterative
process:
$$w_{n}=(1-a_{n}-b_{n})x_{n}+a_{n}z_{n}+b_{n}s_{n} ,\qquad n\geq 1,$$

$$y_{n}=(1-c_{n}-d_{n}-e_{n})x_{n}+c_{n}u_{n}+d_{n}u'_{n}+e_{n}s_{n}' ,\qquad n\geq 1,$$

$$x_{n+1}=(1-\alpha_{n}-\beta_{n}-\gamma_{n})x_{n}+\alpha_{n}v_{n}
+\beta_{n}v'_{n}+\gamma_{n}s_{n}'',\qquad n\geq 1,$$ where
$z_{n},u'_{n}\in T_{1}(x_{n})$ ,  $u_{n},v'_{n}\in T_{2}(w_{n})$
and $v_{n}\in T_{3}(y_{n})$ and\\ $\{a_{n}\}, \{b_{n}\},
\{c_{n}\}, \{d_{n}\}, \{e_{n}\}, \{\alpha_{n}\}, \{\beta_{n}\},
\{\gamma_{n}\}\in [0,1]$  and $\{s_{n}\},\{s'_{n}\}$ and
$\{s''_{n}\}$ are bounded sequence in $E.$
\begin{definition}
A mapping $T:E\to CB(E)$ is said to satisfy condition (I) if
there is a non decreasing function $g:[0,\infty)\to [0,\infty)$
with $g(0)=0$, $g(r)>0$ for $r\in (0,\infty)$ such that
$$dist(x,Tx)\geq g(dist(x,F(T)).$$
Let  $T_{i}:E\to CB(E), (i=1,2,3)$ be three given mappings. The
mappings $T_{1},T_{2},T_{3}$ are said to satisfy condition (II) if
there exist a non decreasing function $g:[0,\infty)\to
[0,\infty)$ with $g(0)=0$, $g(r)>0$ for $r\in (0,\infty)$, such
that
$$\sum _{i=1}^{3}dist(x,T_{i}x)\geq g(dist(x,\mathcal {F})),$$
where $\mathcal {F}=\bigcap _{i=1}^{3}F(T_{i})$.
\end{definition}
\begin{theorem}
Let $E$ be a nonempty closed convex subset of a uniformly convex
Banach space $X$. Let $T_{i}:E \to CB(E),(i=1,2,3) $ be three
multivalued mappings satisfying the condition (C) . Assume that
$\mathcal {F}=\bigcap _{i=1}^{3}F(T_{i})\neq\emptyset$ and
$T_{i}(p)=\{p\}, (i=1,2,3)$ for each $p\in \mathcal {F}.$ Let
$\{x_{n}\}$ be the iterative process defined by (A), and
$a_{n}+b_{n}, c_{n}+d_{n}+e_{n}, \alpha_{n}+\beta_{n}+\gamma_{n}
\in [a,b]\subset(0,1)$ and also
$\sum_{n=1}^{\infty}b_{n}<\infty,\quad
\sum_{n=1}^{\infty}e_{n}<\infty$ and
$\sum_{n=1}^{\infty}\gamma_{n}<\infty$. Assume that $T_{1},T_{2}$
and $ T_{3}$ satisfying the condition (II). Then $\{x_{n}\}$
converges strongly to a common fixed point of $T_{1},T_{2}$ and $
T_{3}$.
\end{theorem}
\begin{proof}
Let $p\in \mathcal {F}.$ Then, by the boundedness  of
$\{s_{n}\},\{s_{n}'\}$ and $\{s_{n}''\}$, we let
$$ M=max\{sup_{n\geq1}\|s_{n}-p\|,sup_{n\geq1}\|s_{n}'-p\|,sup_{n\geq1}\|s_{n}''-p\|\}.$$
 Using (A) and quasi nonexpansiveness of $T_{i}$ (i=1,2,3)
 we have
\begin{multline*}
\parallel w_{n}-p\parallel=\parallel(1-a_{n}-b_{n})x_{n}+a_{n}z_{n}+b_{n}s_{n}-p
\parallel\\
\leq(1-a_{n}-b_{n})\parallel x_{n}-p\parallel+a_{n}\parallel z_{n}-p\parallel+b_{n}\|s_{n}-p\|\\
=(1-a_{n}-b_{n}\parallel x_{n}-p\parallel+a_{n}dist(z_{n},T_{1}(p))+b_{n}\|s_{n}-p\|\\
\leq(1-a_{n}-b_{n})\parallel x_{n}-p\parallel+a_{n}H(T_{1}(x_{n}),T_{1}(p))+b_{n}\|s_{n}-p\|\\
\leq(1-a_{n}-b_{n})\parallel x_{n}-p\parallel+a_{n}\parallel
x_{n}-p\parallel+b_{n}\|s_{n}-p\|\\ \leq
(1-b_{n})\|x_{n}-p\|+b_{n} M\\\leq\|x_{n}-p\|+b_{n}M
\end{multline*}
 and  \begin{multline*} \parallel
y_{n}-p\parallel=\parallel(1-c_{n}-d_{n}-e_{n})x_{n}+c_{n}u_{n}+d_{n}u'_{n}+e_{n}s_{n}'-p
\parallel\\\leq(1-c_{n}-d_{n}-e_{n})
\parallel x_{n}-p\parallel+c_{n}\parallel
u_{n}-p\parallel+d_{n}\parallel
u'_{n}-p\|+e_{n}\|s_{n}'-p\|\\\leq(1-c_{n}-d_{n}-e_{n})\parallel
x_{n}-p\parallel+c_{n}dist(u_{n},T_{2}(p))+d_{n}dist(u'_{n},T_{1}(p))+e_{n}\|s_{n}'-p\|\\
\leq(1-c_{n}-d_{n}-e_{n})\parallel
x_{n}-p\parallel+c_{n}H(T_{2}(w_{n}),T_{2}(p))+d_{n}H(T_{1}(x_{n}),T_{1}(p))+e_{n}\|s_{n}'-p\|\\
\leq(1-c_{n}-d_{n}-e_{n})\parallel x_{n}-p\parallel+c_{n}\parallel
w_{n}-p\parallel+d_{n} \parallel x_{n}-p\parallel+e_{n}\|s_{n}'-p\|\\
\leq(1-c_{n}-d_{n}-e_{n})\parallel
x_{n}-p\parallel+c_{n}\parallel x_{n}-p\parallel+d_{n} \parallel
x_{n}-p\parallel+c_{n}b_{n}M+e_{n}M\\\leq
\|x_{n}-p\|+b_{n}M+e_{n}M .\end{multline*} We also have
\begin{multline} \parallel
x_{n+1}-p\parallel=\parallel(1-\alpha_{n}-\beta_{n}-\gamma_{n})x_{n}
+\alpha_{n}v_{n}+\beta_{n}v'_{n}+\gamma_{n}s_{n}''-p
\parallel\\
\leq(1-\alpha_{n}-\beta_{n}-\gamma_{n})\parallel
x_{n}-p\parallel+\alpha_{n}\parallel
v_{n}-p\parallel+\beta_{n}\parallel
v'_{n}-p\|+\gamma_{n}\|s_{n}''-p\|\\\leq(1-\alpha_{n}-\beta_{n}-\gamma_{n})\parallel
x_{n}-p\parallel+\alpha_{n}dist(v_{n},T_{3}(p))+\beta_{n}dist(v'_{n},T_{2}(p))+\gamma_{n}\|s_{n}''-p\|\\
\leq(1-\alpha_{n}-\beta_{n}-\gamma_{n})\parallel
x_{n}-p\parallel+\alpha_{n}H(T_{3}(y_{n}),T_{3}(p))+\beta_{n}H(T_{2}(w_{n}),T_{2}(p))+\gamma_{n}\|s_{n}''-p\|\\
\leq(1-\alpha_{n}-\beta_{n}-\gamma_{n})\|
x_{n}-p\parallel+\alpha_{n}\parallel y_{n}-p\parallel+\beta_{n}
\parallel w_{n}-p\parallel+\gamma_{n}\|s_{n}''-p\|\\\leq
(1-\alpha_{n}-\beta_{n}-\gamma_{n})\parallel
x_{n}-p\parallel+\alpha_{n}\|x_{n}-p\|
+\alpha_{n}b_{n}M+\alpha_{n}e_{n}M+\beta_{n}\|x_{n}-p\|+\beta_{n}
b_{n} M+\gamma_{n}M \\\leq
(1-\gamma_{n})\|x_{n}-p\|+M(b_{n}+e_{n}+\gamma_{n})\\
=\|x_{n}-p\|+\theta_{n}.\end{multline} where $\theta_{n}=
M(b_{n}+e_{n}+\gamma_{n})$. By assumption we have
$\sum_{n=1}^{\infty}\theta_{n}< \infty.$
 Hence by Lemma 2.6  it follows that $\lim\parallel x_{n}-p\parallel$ exist for any
$p\in \mathcal {F}.$ Since the sequences
 $\{x_{n}\},\{y_{n}\} $ and $\{w_{n}\}$
are bounded, we can find $r>0$ depending on $p$ such that
$x_{n}-p,y_{n}-p,w_{n}-p\in B_{r}(0)$ for all $n\geq 0.$ Denote by
$$N=max\{sup_{n\geq1}\|s_{n}-p\|^{2},sup_{n\geq1}\|s_{n}'-p\|^{2},sup_{n\geq1}\|s_{n}''-p\|^{2}\}.$$
From Lemma 2.7, we get
\begin{multline*}
\parallel w_{n}-p\parallel^{2}=\parallel(1-a_{n}-b_{n})x_{n}+a_{n}z_{n}+b_{n}s_{n}-p
\parallel^{2}\\
\leq(1-a_{n}-b_{n})\parallel x_{n}-p\parallel^{2}+a_{n}\parallel
z_{n}-p\parallel^{2}
+b_{n}\|s_{n}-p\|^{2}-a_{n}(1-a_{n}-b_{n})\varphi(\|x_{n}-z_{n}\|)\\
\leq(1-a_{n}-b_{n}\parallel
x_{n}-p\parallel^{2}+a_{n}dist(z_{n},T_{1}(p))^{2}+b_{n}\|s_{n}-p\|^{2}
-a_{n}(1-a_{n}-b_{n})\varphi(\|x_{n}-z_{n}\|)\\
\leq(1-a_{n}-b_{n})\parallel
x_{n}-p\parallel^{2}+a_{n}H(T_{1}(x_{n}),T_{1}(p))^{2}
+b_{n}\|s_{n}-p\|^{2}-a_{n}(1-a_{n}-b_{n})\varphi(\|x_{n}-z_{n}\|)\\
\leq(1-a_{n}-b_{n})\parallel x_{n}-p\parallel^{2}+a_{n}\parallel
x_{n}-p\parallel^{2}+b_{n}\|s_{n}-p\|^{2}-a_{n}(1-a_{n}-b_{n})\varphi(\|x_{n}-z_{n}\|)\\
\leq (1-b_{n})\|x_{n}-p\|^{2}+b_{n}
N-a_{n}(1-a_{n}-b_{n})\varphi(\|x_{n}-z_{n}\|)\\\leq\|x_{n}-p\|^{2}+b_{n}N-a_{n}(1-a_{n}-b_{n})\varphi(\|x_{n}-z_{n}\|)
\end{multline*}
 It follows from Lemma 2.7 that  \begin{multline*} \parallel
y_{n}-p\parallel^{2}=\parallel(1-c_{n}-d_{n}-e_{n})x_{n}+c_{n}u_{n}+d_{n}u'_{n}+e_{n}s_{n}'-p
\parallel^{2}\\\leq (1-c_{n}-d_{n}-e_{n})
\parallel x_{n}-p\parallel^{2}+c_{n}\parallel
u_{n}-p\parallel^{2}+d_{n}\parallel
u'_{n}-p\|^{2}+e_{n}\|s_{n}'-p\|^{2}\\-\frac{1}{2}(1-c_{n}-d_{n}-e_{n})d_{n}
\varphi(\|x_{n}-u'_{n}\|)-\frac{1}{2}(1-c_{n}-d_{n}-e_{n})c_{n}
\varphi(\|x_{n}-u_{n}\|)\\\leq(1-c_{n}-d_{n}-e_{n})\parallel
x_{n}-p\parallel^{2}+c_{n}dist(u_{n},T_{2}(p))^{2}+d_{n}dist(u'_{n},T_{1}(p))^{2}+e_{n}\|s_{n}'-p\|^{2}\\
-\frac{1}{2}(1-c_{n}-d_{n}-e_{n})d_{n}
\varphi(\|x_{n}-u'_{n}\|)-\frac{1}{2}(1-c_{n}-d_{n}-e_{n})c_{n}
\varphi(\|x_{n}-u_{n}\|)\\
\leq(1-c_{n}-d_{n}-e_{n})\parallel
x_{n}-p\parallel^{2}+c_{n}H(T_{2}(w_{n}),T_{2}(p))^{2}+d_{n}H(T_{1}(x_{n}),T_{1}(p))^{2}+e_{n}\|s_{n}'-p\|^{2}
\\-\frac{1}{2}(1-c_{n}-d_{n}-e_{n})d_{n}
\varphi(\|x_{n}-u'_{n}\|)-\frac{1}{2}(1-c_{n}-d_{n}-e_{n})c_{n}
\varphi(\|x_{n}-u_{n}\|)\\
\leq(1-c_{n}-d_{n}-e_{n})\parallel
x_{n}-p\parallel^{2}+c_{n}\parallel w_{n}-p\parallel^{2}+d_{n}
\parallel x_{n}-p\parallel^{2}+e_{n}\|s_{n}'-p\|^{2}\\-\frac{1}{2}(1-c_{n}-d_{n}-e_{n})d_{n}
\varphi(\|x_{n}-u'_{n}\|)-\frac{1}{2}(1-c_{n}-d_{n}-e_{n})c_{n}
\varphi(\|x_{n}-u_{n}\|)\\
\leq(1-c_{n}-d_{n}-e_{n})\parallel
x_{n}-p\parallel^{2}+c_{n}\parallel x_{n}-p\parallel^{2}+d_{n}
\parallel x_{n}-p\parallel^{2}+c_{n}b_{n}N+e_{n}N\\-\frac{1}{2}(1-c_{n}-d_{n}-e_{n})d_{n}
\varphi(\|x_{n}-u'_{n}\|)-\frac{1}{2}(1-c_{n}-d_{n}-e_{n})c_{n}
\varphi(\|x_{n}-u_{n}\|)\\\leq
\|x_{n}-p\|^{2}+b_{n}N+e_{n}N\\-\frac{1}{2}(1-c_{n}-d_{n}-e_{n})d_{n}
\varphi(\|x_{n}-u'_{n}\|)-\frac{1}{2}(1-c_{n}-d_{n}-e_{n})c_{n}
\varphi(\|x_{n}-u_{n}\|) .\end{multline*} By another application
of Lemma 2.7 we obtain that
\begin{multline*} \parallel
x_{n+1}-p\parallel^{2}=\parallel(1-\alpha_{n}-\beta_{n}-\gamma_{n})x_{n}
+\alpha_{n}v_{n}+\beta_{n}v'_{n}+\gamma_{n}s_{n}''-p
\parallel^{2}\\
\leq(1-\alpha_{n}-\beta_{n}-\gamma_{n})\parallel
x_{n}-p\parallel^{2}+\alpha_{n}\parallel
v_{n}-p\parallel^{2}+\beta_{n}\parallel
v'_{n}-p\|^{2}+\gamma_{n}\|s_{n}''-p\|^{2}\\
\alpha_{n}(1-\alpha_{n}-\beta_{n}-\gamma_{n})\varphi(\|x_{n}-v_{n}\|)
\\\leq(1-\alpha_{n}-\beta_{n}-\gamma_{n})\parallel
x_{n}-p\parallel^{2}+\alpha_{n}dist(v_{n},T_{3}(p))^{2}+\beta_{n}dist(v'_{n},T_{2}(p))^{2}
+\gamma_{n}\|s_{n}''-p\|^{2}\\-\alpha_{n}(1-\alpha_{n}-\beta_{n}-\gamma_{n})\varphi(\|x_{n}-v_{n}\|)\\
\leq(1-\alpha_{n}-\beta_{n}-\gamma_{n})\parallel
x_{n}-p\parallel^{2}+\alpha_{n}H(T_{3}(y_{n}),T_{3}(p))^{2}+\beta_{n}H(T_{2}(w_{n}),T_{2}(p)^{2}+\gamma_{n}\|s_{n}''-p\|\\
-\alpha_{n}(1-\alpha_{n}-\beta_{n}-\gamma_{n})\varphi(\|x_{n}-v_{n}\|)\\
\leq(1-\alpha_{n}-\beta_{n}-\gamma_{n})\|
x_{n}-p\parallel^{2}+\alpha_{n}\parallel
y_{n}-p\parallel^{2}+\beta_{n}
\parallel w_{n}-p\parallel^{2}+\gamma_{n}\|s_{n}''-p\|^{2}\\-
 \alpha_{n}(1-\alpha_{n}-\beta_{n}-\gamma_{n})\varphi(\|x_{n}-v_{n}\|)\\\leq
(1-\alpha_{n}-\beta_{n}-\gamma_{n})\parallel
x_{n}-p\parallel^{2}+\alpha_{n}\|x_{n}-p\|^{2}
+\alpha_{n}b_{n}N+\alpha_{n}e_{n}N+\beta_{n}\|x_{n}-p\|^{2}+\beta_{n}
b_{n}
N+\gamma_{n}N\\-\alpha_{n}(1-\alpha_{n}-\beta_{n}-\gamma_{n})\varphi(\|x_{n}-v_{n}\|)
-\frac{1}{2}\alpha_{n}(1-c_{n}-d_{n}-e_{n})d_{n}
\varphi(\|x_{n}-u'_{n}\|)\\-\frac{1}{2}\alpha_{n}(1-c_{n}-d_{n}-e_{n})c_{n}
\varphi(\|x_{n}-u_{n}\|)
-a_{n}\beta_{n}(1-a_{n}-b_{n})\varphi(\|x_{n}-z_{n}\|)\\\leq
\|x_{n}-p\|^{2}+N(b_{n}+e_{n}+\gamma_{n})
-\alpha_{n}(1-\alpha_{n}-\beta_{n}-\gamma_{n})\varphi(\|x_{n}-v_{n}\|)
-\frac{1}{2}\alpha_{n}(1-c_{n}-d_{n}-e_{n})d_{n}
\varphi(\|x_{n}-u'_{n}\|)\\-\frac{1}{2}\alpha_{n}(1-c_{n}-d_{n}-e_{n})c_{n}
\varphi(\|x_{n}-u_{n}\|)
-a_{n}\beta_{n}(1-a_{n}-b_{n})\varphi(\|x_{n}-z_{n}\|).\end{multline*}
So, we have
\begin{multline*} \frac
{1}{2}a^{2}(1-b)\varphi(\|x_{n}-u'_{n}\|)\\\leq\frac{1}{2}
\alpha_{n}
(1-c_{n}-d_{n}-e_{n})d_{n}\varphi(\|x_{n}-u'_{n}\|)\\\leq\parallel
x_{n}-p\parallel^{2}-\parallel
x_{n+1}-p\parallel^{2}+N(b_{n}+e_{n}+\gamma_{n}).\end{multline*}
This implies that
$$\sum_{n=1}^{\infty}{a^{2}(1-b)\varphi(\|x_{n}-u'_{n}\|)}\leq\|x_{1}-p\|^{2}
+\sum_{n=1}^{\infty}N(b_{n}+e_{n} +\gamma_{n})<\infty$$ from
which it follows that $\lim
_{n\to\infty}\varphi(\|x_{n}-u'_{n}\|)=0.$  Since $\varphi$ is
continuous at $0$ and is strictly increasing, we have
$$\lim_{n\to\infty}\|x_{n}-u'_{n}\|=0.$$
Similarly we obtain that
$$\lim_{n\to\infty}\|x_{n}-z_{n}\|=\lim_{n\to\infty}\|x_{n}-u_{n}\|=\lim_{n\to\infty}\|x_{n}-v_{n}\|=0.$$
Hence we obtain $dist(x_{n},T_{1}x_{n})\leq\|x_{n}-u'_{n}\|\to 0$
as $n\to\infty.$ Also we have
$$\lim_{n\to\infty}\|x_{n}-w_{n}\|=\lim_{n\to\infty}(a_{n}\|z_{n}-x_{n}\|+b_{n}\|s_{n}-x_{n}\|)=0.$$
and
$$\lim_{n\to\infty}\|x_{n}-y_{n}\|=\lim_{n\to\infty}(c_{n}\|u_{n}-x_{n}\|+d_{n}\|u'_{n}-x_{n}\|+e_{n}\|s'_{n}-x_{n}\|)=0.$$

 Therefore by Lemma 2.5 we
have\begin{multline*}dist(x_{n},T_{2}(x_{n}))\leq dist(x_{n},T_{2}(w_{n}))+H(T_{2}(w_{n}),T_{2}(x_{n}))\\
\leq
dist(x_{n},T_{2}(w_{n}))+2\,dist(w_{n},T_{2}(w_{n}))+\|x_{n}-w_{n}\|\\\leq
3\,\|x_{n}-w_{n}\|+3\,dist(x_{n},T_{2}(w_{n}))\\\leq
3\,\|x_{n}-w_{n}\|+3\,\|x_{n}-u_{n}\|\to 0\qquad as\quad
n\to\infty.\end{multline*}and
\begin{multline*}dist(x_{n},T_{3}x_{n})\leq dist(x_{n},T_{3}(y_{n}))+H(T_{3}(y_{n}),T_{3}(x_{n}))\\
\leq
dist(x_{n},T_{3}(y_{n}))+2\,dist(y_{n},T_{3}(w_{n}))+\|x_{n}-y_{n}\|\\\leq
3\,\|x_{n}-y_{n}\|+3\,dist(x_{n},T_{3}(y_{n}))\\\leq
3\,\|x_{n}-y_{n}\|+3\,\|x_{n}-v_{n}\|\to 0\qquad as\quad
n\to\infty.\end{multline*}
 Note that by
our assumption $\lim_{n\to\infty}dist(x_{n},\mathcal {F})=0$.
Hence there exist a subsequence $\{x_{n_{k}}\}$ of $\{x_{n}\}$
and a sequence $\{p_{k}\}$ in $\mathcal {F}$ such that
$\|x_{n_{k}}-p_{k}\|<\frac{1}{2^k}$ for all $k$. Therefore  by
inequality 3.1 we get
\begin{multline*}
\|x_{n_{k+1}}-p\|\leq\|x_{n_{k+1}-1}-p\|+\theta_{n_{k+1}-1}\\\leq
\|x_{n_{k+1}-2}-p\|+\theta_{n_{k+1}-2}+\theta_{n_{k+1}-1}\\\leq...
\\\leq
\|x_{n_{k}}-p\|+\sum_{i=1}^{n_{k+1}-n_{k}-1}\theta_{n_{k}+i}
\end{multline*} For all $p\in \mathcal {F}.$ This implies that
\begin{multline*}\|x_{n_{k+1}}-p\|\leq \|x_{n_{k}}-p_{k}\|+\sum_{i=1}^{n_{k+1}-n_{k}-1}\theta_{n_{k}+i}\\
\leq\frac{1}{2^k}+\sum_{i=1}^{n_{k+1}-n_{k}-1}\theta_{n_{k}+i}.
\end{multline*}
Now, we show that$\{p_{k}\}$ is Cauchy sequence in $E$. Note that
\begin{multline*}\|p_{k+1}-p_{k}\|\leq
\|p_{k+1}-x_{n_{k+1}}\|+\|x_{n_{k+1}}-p_{k}\|\\
<\frac{1}{2^{k+1}}+\frac{1}{2^k}+\sum_{i=1}^{n_{k+1}-n_{k}-1}\theta_{n_{k}+i}
\\<\frac{1}{2^{k-1}}+\sum_{i=1}^{n_{k+1}-n_{k}-1}\theta_{n_{k}+i}.
\end{multline*} This implies that $\{p_{k}\}$ is Cauchy sequence in
$E$ and hence converges to $q\in E.$ Since for $i=1,2,3$
$$ dist(p_{k},T_{i}(q))\leq H(T_{i}(p_{k}),T_{i}(q))\leq \|p_{k}-q\|$$
and $p_{k}\to q$ as $n\to\infty$, it follows that $
dist(q,T_{i}(q))=0$ and thus $q\in \mathcal {F}$  and
$\{x_{n_{k}}\} $ converges strongly to $q$. Since
$\lim_{n\to\infty}\| x_{n}-q\|$ exists, we conclude that
$\{x_{n}\}$ converges strongly to $q$.
\end{proof}

\begin{theorem}
Let $E$ be a nonempty compact convex subset of uniformly convex
Banach space $X$. Let $T_{i}:E \to CB(E),(i=1,2,3) $ be three
multivalued mappings satisfying the condition (C) . Assume that
$\mathcal {F}=\bigcap _{i=1}^{3}F(T_{i})\neq\emptyset$ and
$T_{i}(p)=\{p\}, (i=1,2,3)$ for each $p\in \mathcal {F}.$ Let
$\{x_{n}\}$ be the iterative process defined by (A), and
$a_{n}+b_{n},c_{n}+d_{n}+e_{n},\alpha_{n}+\beta_{n}+\gamma_{n}
\in [a,b]\subset(0,1)$ and also
$\sum_{n=1}^{\infty}b_{n}<\infty,\quad
\sum_{n=1}^{\infty}e_{n}<\infty$ and
$\sum_{n=1}^{\infty}\gamma_{n}<\infty$. Then $\{x_{n}\}$
converges strongly to a common fixed point of $T_{1},T_{2}$ and $
T_{3}$.\end{theorem}
\begin{proof}
As in the proof of Theorem  3.2, we have
$\lim_{n\to\infty}dist(T_{i}(x_{n}),x_{n})=0, (i=1,2,3).$ Since
$E$ is compact, there exists a subsequence $\{x_{n_{k}}\}$ of
$\{x_{n}\}$ such that $\lim x_{n_{k}}=w$ for some $w\in E.$ By
lemma 2.6, for $i=1,2,3$  we have \begin{multline*}
dist(w,T_{i}(w))\leq \|w-x_{n_{k}}\|+ dist
(x_{n_{k}},T_{i}(w))\\\leq
\|w-x_{n_{k}}\|+dist(x_{n_{k}},T_{i}(x_{n_{k}}))+H(T_{i}(x_{n_{k}}),T_{i}(w))\\\leq
3dist(x_{n_{k}},T_{i}(x_{n_{k}}))+2 \|w-x_{n_{k}}\|\to 0\qquad
as\quad k\to\infty,\end{multline*} this implies that
 $w\in \mathcal {F} $ . Since
$\{x_{n_{k}}\}$ converges strongly to $w$ and \\$\lim_{n\to\infty}
\|x_{n}-w\|$ exist (as in thre proof of Theorem 3.2), this
implies that $\{x_{n}\}$ converges strongly to $w.$\end{proof}

We now intend to remove the restriction that $T_{i}(p)=p$ for each
$p\in \mathcal{F}.$
We define the following iteration process.\\
{\bf(B)}:\quad Let $X$ be a Banach space, $E$ be a nonempty convex
subset of $X$ and \\ $T_{i}:E\to P(E),(i=1,2,3)$ be given
mappings and
$$P_{T_{i}}(x)=\{y\in T_{i}(x) : \parallel x-y\parallel=dist (x,T_{i}(x))\}.$$ Then, for $x_{1}\in E$, we consider the following
iterative process:
$$w_{n}=(1-a_{n}-b_{n})x_{n}+a_{n}z_{n}+b_{n}s_{n} ,\qquad n\geq 1,$$

$$y_{n}=(1-c_{n}-d_{n}-e_{n})x_{n}+c_{n}u_{n}+d_{n}u'_{n}+e_{n}s_{n}' ,\qquad n\geq 1,$$

$$x_{n+1}=(1-\alpha_{n}-\beta_{n}-\gamma_{n})x_{n}+\alpha_{n}v_{n}
+\beta_{n}v'_{n}+\gamma_{n}s_{n}'',\qquad n\geq 1,$$ where
$z_{n},u'_{n}\in P_{T_{1}}(x_{n})$ ,  $u_{n},v'_{n}\in
P_{T_{2}}(w_{n})$ and $v_{n}\in P_{T_{3}}(y_{n})$ and\\
$\{a_{n}\}, \{b_{n}\}, \{c_{n}\}, \{d_{n}\}, \{e_{n}\},
\{\alpha_{n}\}, \{\beta_{n}\}, \{\gamma_{n}\}\in [0,1]$  and
$\{s_{n}\},\{s'_{n}\}$ and $\{s''_{n}\}$ are bounded sequence in
$E.$
\begin{theorem}
Let $E$ be a nonempty closed convex subset of a uniformly convex
Banach space $X$. Let $T_{i}:E \to P(E),(i=1,2,3)$ be multivalued
mappings such that  $P_{T_{i}} $ satisfing the condition (C). Let
$\{x_{n}\}$ be the iterative process defined by (B), and
$a_{n}+b_{n},c_{n}+d_{n}+e_{n},\alpha_{n}+\beta_{n}+\gamma_{n}
\in [a,b]\subset(0,1)$ and also
$\sum_{n=1}^{\infty}b_{n}<\infty,\quad
\sum_{n=1}^{\infty}e_{n}<\infty$ and
$\sum_{n=1}^{\infty}\gamma_{n}<\infty$. Assume that $T_{1},T_{2}$
and $ T_{3}$ satisfying the condition (II) and
$\mathcal{F}\neq\emptyset.$ Then $\{x_{n}\}$ converges strongly
to a common fixed point of $T_{1},T_{2}$ and $ T_{3}$.
\end{theorem}
\begin{proof}
Let $p\in \mathcal{F}.$ Then, for $i=1,2,3$ we have  $p\in
P_{T_{i}}(p)=\{p\}$ . Also, we have
 $$ \|z_n-p\|\leq dist(z_n,P_{T_{1}}(p))\leq
 H(P_{T_{1}}(x_{n}),P_{T_{1}}(p))\leq \|x_{n}-p\|$$
and
$$ \|u_n-p\|\leq dist(u_n,P_{T_{2}}(p))\leq
 H(P_{T_{2}}(w_{n}),P_{T_{2}}(p))\leq \|w_{n}-p\|,$$
and
$$ \|v_n-p\|\leq dist(v_n,P_{T_{3}}(p))\leq
 H(P_{T_{3}}(y_{n}),P_{T_{3}}(p))\leq \|y_{n}-p\|.$$
 Now, by similar argument as in the proof of Theorem 3.2, $\lim_{n\to\infty}\| x_{n}-q\|$  exists. Also we get
 a sequence  $\{p_{k}\}\in \mathcal{F}$ which converges to
some $q\in E$. Since for each $i=1,2,3$
$$dist(p_{k},T_{i}(q))\leq dist(p_{k},P_{T_{i}}(q))\leq H(P_{T_{i}}(p_{k}),P_{T_{i}}(q))\leq\parallel q-p_{k}\parallel,$$
and $p_{k}\to q$ as $k\to \infty$, it follows that
$dist(q,T_{i}(q))=0$ for $i=1,2,3$. Hence $q\in\mathcal{F} $ and
$\{x_{n_{k}}\}$ converges strongly to $q.$  Since
$\lim_{n\to\infty}\| x_{n}-q\|$ exists, we conclude that
$\{x_{n}\}$ converges strongly to $q.$
\end{proof}

\bibliographystyle{amsplain}

\end{document}